\documentclass[reqno]{amsart}
\usepackage{amsmath,amsthm,amssymb}
\usepackage{latexsym}
\usepackage{eucal}

\usepackage{amsfonts}
\usepackage{mathtools}
\usepackage [english]{babel}

\let\Re=\undefined
\DeclareMathOperator{\Re}{Re}
\let\Im=\undefined
\DeclareMathOperator{\Im}{Im}

\def\~{\widetilde}

\newtheorem{theorem}{Theorem}[section]
\newtheorem{lemma}{Lemma}[section]

\newtheorem{proposition}{Proposition}[section]

\let\Re=\undefined
\DeclareMathOperator{\Re}{Re}
\let\Im=\undefined
\DeclareMathOperator{\Im}{Im}

\def\~{\widetilde}

\begin{document}
\title{Randomized Verblunsky Parameters in Steklov's Problem}
\author{Keith Rush}
\address{
\begin{flushleft}
Milwaukee Brewers\\  Strategy and Analytics\\
1 Brewers Way, Milwaukee, WI, 53214, USA\\  keith.rush@brewers.com\\
\end{flushleft}
}
\maketitle
\begin{abstract}
	We consider randomized Verblunsky parameters for orthogonal polynomials on the unit circle as they relate to the problem of Steklov, bounding the polynomials' uniform norm independent of $n$. 
\end{abstract}
\section{Steklov Problems in Orthogonal Polynomials}
Let $\mu$ be a probability measure on $\mathbb{T} = \{z \in \mathbb{C} : |z|=1\}$. Define $\{\Phi_n\}$ to be the unique polynomials satisfying
$$\int_{\mathbb{T}} \Phi_n(z) z^{-j} d\mu = 0, \quad 0 \leq j < n, \quad \quad \quad \text{coeff}(\Phi_n, n)=1.$$
We call these $\{\Phi_n\}$ the monic orthogonal polynomials.

Denote by $\phi_n$ the orthonormal polynomials,
$$\phi_n = \frac{\Phi_n}{\|\Phi_n\|_{L^2(d\mu)}}.$$

A central question about orthogonal polynomials concerns their asymptotic behavior as $n \to \infty$. Szeg\H{o} proved $L^2(d\mu)$ convergence of the orthonormal polynomials to a particular outer function in $H^2(\mathbb{D})$, the Szeg\H{o} function, given Szeg\H{o}'s condition
$$\int_{\mathbb{T}} \log w(\theta) d\theta > -\infty$$
for $w(\theta)d\theta = d\mu_{ac}$. Under the same condition, the orthonormal and monic orthogonal polynomials are uniformly comparable in $n$. 

There has been much work on similar conditions. V.A. Steklov conjectured in \cite{steklov:ums21} that if $\mu \in S_\delta$,
$$S_\delta = \left\{\mu \in \mathcal{M}(\mathbb{T}), \quad \int_{\mathbb{T}} d\mu =1,\quad d\mu' \geq \delta \text{ a.e.}\right\},$$
the orthonormal polynomials $\phi_n(z; d\mu)$ would obey the bound
$$\|\phi_n\|_{L^\infty(\mathbb{T})} \leq C_\delta.$$
This conjecture was disproved by E.A. Rahmanov in \cite{rahmanov:osc80}, who constructed polynomials with logarithmic growth of the uniform norm in $n$ via an algebraic identity.

Rahmanov's result sparked interest in the rate of polynomial growth as measured by the uniform norm. A simple argument (see \cite{geronimus:poc60} for example) shows
$$\|\phi_n\|_{L^\infty(\mathbb{T})} = o(\sqrt{n}).$$
Rahmanov nearly matched this growth in \cite{rahmanov:ego82} where he showed, for 
$$M_{n, \delta} \coloneqq \sup_{\mu \in S_\delta} \|\phi_n(z; d\mu)\|_{L^\infty(\mathbb{T})}$$
that
$$\left(\frac{n}{\ln^3n}\right)^{1/2} <_\delta M_{n, \delta} <_\delta n^{1/2}.$$
This was the sharpest quantification of $M_{n, \delta}$ until \cite{aptekarev:ops14}, where Aptekarev, Denisov and Tulyakov proved 
$$M_{n, \delta} \sim_\delta n^{1/2}.$$

Thus the Steklov condition is insufficient to break the barrier of $n^{1/2}$. We are then led to the natural question: are there any related conditions for which something stronger can be proved?

This question was answered in the affirmative in \cite{denisov:tgp16, denisov:opc16}, which imposed $w, w^{-1} \in L^\infty(\mathbb{T})$ and $w, w^{-1} \in \text{BMO}(\mathbb{T})$, respectively, although both consider only polynomials orthogonal with respect to absolutely continuous measures. Both conditions are sufficient to break the $n^{1/2}$ barrier, and lower bounds were established in \cite{denisov:tgp16} showing that its results are sharp in some regimes. Ambroladze showed in \cite{ambroladze:opr92} that polynomials orthogonal with respect to a continuous weight function may still grow in uniform norm, and the author in his thesis proved upper bounds in terms of the weight function's modulus of continuity which are also sharp in some cases.

Thus the conjecture of Steklov turned out to be quite incorrect, and in fact much stronger conditions are still insufficient to bound the polynomials' uniform norm. On the other hand, the conjecture remained open for over 50 years, and survived many attacks (see \cite{ geronimus:ocv62, geronimus:rbo63, geronimus:seo77, golinskii:tvs74} as well as the survey \cite{suetin:vsp77}). Empirically, counterexamples seem to be sparse. This suggests the probabilistic Steklov question: if we take a random measure, what can we say?

Consider a fundamental system in orthogonal polynomials on the unit circle (OPUC), the so-called Szeg\H{o} recursion:
\begin{equation}
\label{recur}
\Phi_{n+1}(z) = z\Phi_n(z) - \overline{\alpha_n}\Phi_n^*(z)$$
$$
\Phi_{n+1}^*(z) = \Phi_n^*(z) - \alpha_nz\Phi_n(z)
\end{equation}
for $\Phi_n^*(z) = z^n\overline{\Phi_n(z)}, z \in \mathbb{T}$, some $\alpha_n \in \mathbb{D}$. These $\alpha_n$ are one domain in which the measure can be parameterized. Indeed, sequences $\{\alpha_n\} \in \mathbb{D}^\infty$ correspond bijectively to probability measures on the circle with infinitely many points of support (Verblunsky's Theorem \cite{verblunsky:oph36}). 

\eqref{recur} suggests a natural randomization in OPUC. Let $\{a_n\}_{n=0}^\infty$ be a fixed (real) sequence, with $|a_n|<1$ for each $n$.
Let $\{\omega_n\}_{n=0}^\infty$ be an independent sequence of complex-valued random variables bounded by $1$ in absolute value and with expectation 0. For simplicity we also assume rotational symmetry of the $\omega_n$.

Let $\alpha_n = a_n\omega_n$, so that $\{\alpha_n\}_{n=0}^\infty$ is a sequence of independent random variables with the same decay properties as $\{a_n\}$. Denote by $dS$ the probability measure thus defined on $\mathbb{D}^\infty$, under the $\sigma$-algebra generated by $\{\alpha_n\}$.

Let 
$$\mathcal{F}_{n} = \sigma(\alpha_0, \dots, \alpha_{n-1})$$
define a filtration on $\mathbb{D}^\infty$. Since the $\alpha_j$ are independent random variables, we have
$$\alpha_j \perp \mathcal{F}_n \text{ for }j \geq n, \quad \quad \Phi_n(z), \Phi_n^*(z) \in \mathcal{F}_n.$$

Therefore by \eqref{recur}
\begin{equation}
\label{MG}
\mathbb{E}[\Phi_{j}^*(z) | \mathcal{F}_n] = \Phi_n^*(z) \quad \quad j \geq n
\end{equation}
and the polynomials $\Phi_n^*$ have martingale structure.

We expect this structure will enable us to provide some quantitative control on the polynomials, and the relevant question becomes:

How much decay must we impose on $\{a_n\}$ so that the associated orthonormal polynomials $\phi_n^*$ remain bounded in $L^{\infty}(\mathbb{T})$ with high probability?

A well-known analogy relates orthogonal polynomials and partial sums of Fourier series. By considering a particular change of variables (Pr\"{u}fer variables, see \cite{kiselev:mpe98}), the polynomials can be thought of as a nonlinear analogue of these partial sums, where the Verblunsky parameters $\{\alpha_n\}$ play the role of the Fourier coefficients.

Thus our question is analogous to those answered by Salem and Zygmund in the classical paper \cite{salem:spt54}. We expect similar results to be provable. In the main result of this paper, we prove a statement similar to one of Salem and Zygmund.

\begin{theorem}
\label{polynomial-SZ}
Let $\{a_j\} \in (-1, 1)^\infty$ be fixed. Let
$$R_k = \sum_{n=k}^\infty a_n^2$$
and assume 
$$\sum_n \frac{\sqrt{R_n}}{n\sqrt{\log n}} < \infty.$$

Let
$\alpha_n = a_n\omega_n$
with $\{\omega_n\}$ a sequence of independent complex random variables, bounded by 1 in absolute value and rotationally symmetric in $\mathbb{C}$.
Then with probability 1 there exists a random constant $C$ such that
$$\sup_n\|\Phi_n^*\|_\infty \leq C.$$
\end{theorem}

{\bf Remark 1.} Since our conditions on $\{a_n\}$ imply in particular $\{\alpha_n\} \in \ell^2\left(\mathbb{Z}_+\right)$ and therefore Szeg\H{o}'s condition is satisfied, the orthonormal and monic orthogonal polynomials are uniformly comparable in $n$. We use the monic polynomials due to the simplicity of the recurrence \eqref{recur} under this normalization.

{\bf Remark 2.} In Salem and Zygmund \cite{salem:spt54}, the following is proved.

\begin{theorem} [Salem-Zygmund \cite{salem:spt54}, Theorem 5.1.5]
Let $R_n=\sum_{m=n+1}^\infty a_m^2$. If
$$\sum_n \frac{\sqrt{R_n}}{n\sqrt{\log n}} <\infty$$
then the series 
$$\sum_{m=1}^\infty a_m\psi_m(t)\cos mx$$
represents a continuous function for almost every value of $t$, where $\{\psi_n\}$ is the Rademacher system.
\end{theorem}

Our result is analogous to theirs. Additionally, they showed that their condition was the best possible of its kind. We are able to prove the same.

\begin{theorem}
	\label{sharpness}
	Let $\nu(n) \to \infty$ monotonically with $n$. Then there is a sequence $\{a_j\} \in (-1, 1)^\infty$ and $\{\omega_n\}$ a sequence of independent, rotationally symmetric $\mathbb{D}$-valued random variables so that with probability 1 there is $\theta^* \in [0, 2\pi)$ satisfying $$\sup_n|\Phi_n(e^{i\theta^*})|=\infty$$
	and, moreover, for $R_n$ as in Theorem \ref{polynomial-SZ},
 $$\sum_n \frac{\sqrt{R_n}}{\nu(n)n\sqrt{\log n}} <\infty.$$

\end{theorem}

These results are similar to those found in \cite{chhaibi:omc16}, which also considers randomized Verblunsky parameters in OPUC, but from a different perspective and answering different questions. I am indebted to \cite{chhaibi:omc16} for the idea of  working on the linear level, which allowed the results in this paper to be sharp in the sense noted, and simplified the arguments.

A few remarks on notation and definitions. 

We will use $f \lesssim g $ to denote the existence of a universal constant $C$ so that $f(x) \leq Cg(x)$ for all values of the argument.  $f(\delta,x) <_\delta g(\delta,x)$ will denote the existence of a $\delta$-dependent constant $C_\delta$ so that $f(\delta, x) \leq C_\delta g(\delta,x)$ for all $\delta>0$. $\alpha \ll \beta$ will indicate the existence of a sufficiently large constant $A$ so that $\frac{\beta}{\alpha} \geq A$. $\gg$ is defined similarly.

Infinite sequences will appear often in this paper so we will repeatedly suppress their indices. Unless otherwise stated, all sequences will be assumed to run from $0$ to $\infty$.

$\Omega$ will always refer to $\mathbb{D}^\infty$ considered as the state space for the random sequence $\{\alpha_n\}$. The probability measure on $\mathbb{D}^\infty$ is dependent upon choices of $\{a_n\}$ and $\{\omega_n\}$ as above. $\omega\in \Omega$ refers to an element of this state space, i.e. a particular sequence of Verblunsky parameters $\{\alpha_n\}$, with probability inherited by the setup of $\{a_n\}$ and $\{\omega_n\}$.

 Since $|\Phi_n(z)| = |\Phi_n^*(z)|$ for $z \in \mathbb{T}$, the quantity $\frac{z\Phi_n}{\Phi_n^*}$ is a phase if $z \in \mathbb{T}$. Furthermore, this quantity is continuous in $\theta$ for $z=e^{i\theta}$ since $\Phi_n^*$ does not vanish in $\overline{\mathbb{D}}$. We will denote
 $$e^{i\gamma_n} \coloneqq \frac{z\Phi_n}{\Phi_n^*}.$$
 This $\gamma_n$ depends on $\theta$ as well as $\{\alpha_j\}_{j=0}^{n-1}$. The dependence will be at times implicit but clear from context. When appropriate we will call attention to the dependence of $\gamma_n$ on $\theta$ by writing $\gamma_n(\theta)$. Note that $\gamma_n$ is measurable with respect to $\mathcal{F}_{n}$ due to the similar measurability of $\Phi_n$, $\Phi_n^*$. Note further that $e^{i\gamma_n(\theta)}$ may wind.

Other notations will be needed and defined throughout the text.

\section{Basic lemmas, upper bound}

We use the Markovian nature of the sequence $\{\Phi_n^*\}$ to decouple in a sparse manner and its martingale nature to control the transfer process uniformly. In this section we collect the lemmas which will be needed to prove Theorem \ref{polynomial-SZ} above. 

Let $\{\alpha_n\}_{n=0}^\infty$ be randomized as in the statement of the Theorem, and the filtration $\{\mathcal{F}_k\}_{k=0}^\infty$ be as defined in the previous section. The polynomial recursion says
$$\Phi_{n+1}^*(z) = \Phi_n^*(z)\left(1 - \frac{\alpha_nz\Phi_n}{\Phi_n^*}\right).$$
Therefore
\begin{equation}\label{product-MG-structure}
\Phi_{n+1}^*(z) = \Phi_k^*(z)\prod_{j=k}^n\left(1 - \frac{\alpha_jz\Phi_j}{\Phi_j^*}\right).
\end{equation}

Instead of working on this multiplicative level and dealing with the associated nonlinearity, we will consider the logarithm of the polynomial process as in \cite{chhaibi:omc16}. \eqref{product-MG-structure} shows
\begin{equation}\label{linear-MG-structure}
\log \Phi_{n+1}^* = \sum_{j=k}^n \log \left(1-\frac{\alpha_j z \Phi_j}{\Phi_j^*} \right) + \log \Phi_k^* = \sum_{j=k}^n \log\left(1-\alpha_je^{i\gamma_j(\theta)}\right) + \log \Phi_k^*.
\end{equation}

Taking the conditional expectation of \eqref{linear-MG-structure} we have
$$\mathbb{E} \left[ \log \Phi_{n+1}^* | \mathcal{F}_k\right] = \sum_{j=k}^n\mathbb{E} \left[ \log\left( 1-\alpha_je^{i\gamma_j(\theta)} \right) \bigg| \mathcal{F}_k\right] + \mathbb{E}\left[\log \Phi_k^*\bigg| \mathcal{F}_k\right]$$
$$= \sum_{j=k}^n\mathbb{E} \left[ \sum_{m=1}^\infty -\frac{(\alpha_j e^{i\gamma_j})^m}{m} \bigg| \mathcal{F}_k\right] + \log \Phi_k^* = \sum_{j=k}^n \sum_{m=1}^\infty \mathbb{E} \left[ -\frac{(\alpha_j e^{i\gamma_j})^m}{m} \bigg| \mathcal{F}_k\right] + \log \Phi_k^* $$
$$= \log \Phi_k^* + O\left(\sum_{s=0}^\infty |\alpha_s|^2\right). $$

The last equality follows by noting $e^{i\gamma_j} \in \mathcal{F}_{j}$ but $\alpha_j \perp \mathcal{F}_{j}$, and using the tower property of conditional expectation. To recover the martingale structure, we reduce to working with the first term of the Taylor expansion.

\begin{lemma}\label{reduction_to_sum}
	Let $G(k) = 2^{2^k}$. To prove Theorem \ref{polynomial-SZ}, it is sufficient to prove that in the same setup, for $\{\theta_j\}_{j=1}^{G(k+1)}$ the $G(k+1)$-st roots of unity and $\{\lambda_k\}$ some summable sequence, the sequence of events
	$$E_k \coloneqq \bigcup_{j=1}^{G(k+1)} \left\{ \max_{G(k) \leq s \leq G(k+1)} \left|\sum_{m=G(k)}^s \alpha_m e^{i \gamma_m(\theta_j)}\right| \geq \lambda_k\right\}$$ occurs only finitely many times with probability 1.
\end{lemma}
\begin{proof}
	Let $\left\{\alpha_n\right\} \in \Omega_0$ where $\mathbb{P}\left[\Omega_0\right]=1, \Omega_0 = \left\{ \omega \in \Omega: E_k \text{ occurs only finitely many times}\right\}$. Let $l$ be large enough so that $|\alpha_m| \leq \frac{1}{2}$ for $m \geq l$, and let $r \geq l$. Then 
$$|\log \Phi_r^*| \leq \left|\sum_{j=l}^r \log \left(1-\alpha_je^{i\gamma_j}\right)\right|+|\log \Phi_l^*| = \left|\sum_{j=l}^r\sum_{m=1}^\infty -\frac{(\alpha_j e^{i\gamma_j})^m}{m}\right|+|\log \Phi_l^*|$$
$$
 \leq \left| \sum_{j=l}^r \alpha_je^{i\gamma_j}\right| + C\sum_{j=l}^{r} |\alpha_m|^2 + |\log \Phi_l^*|.$$

 So for $s \geq G(k)$, $k \gg 1$ and any $j \in \{1, \dots, G(k+1)\}$,
\begin{equation}\label{inductive_sum}
\left|\log \Phi_s^*(e^{i\theta_j})\right| \leq \left|\sum_{m=G(k)}^s \alpha_me^{i\gamma_m(\theta_j)}\right|+O\left(\sum_{m=G(k)}^s |\alpha_m|^2\right) + \log|\Phi_{G(k)}^*|.
\end{equation}
Since $\left\{\alpha_n \right\} \in \Omega_0$ \eqref{inductive_sum} shows, for sufficiently large  $k$,
\begin{equation}  \label{G_k_control} \max_{G(k) \leq s \leq G(k+1)} \left|\log \Phi_s^*(\theta_j)\right| \leq \sum_{j \leq k} \lambda_j + \sum_{j \leq k} C\left(\sum_{m=G(j)}^{G(j+1)} |\alpha_m|^2\right) + C \leq C\end{equation} 
where $C$ is a random constant, and the last inequality uses the summability of $\{\lambda_k\}$ and $\{|\alpha_m|^2\}$.

Therefore  $$\max_{G(k) \leq s \leq G(k+1)} |\Phi_s^*(\theta_j)| \leq C$$ uniformly where $\theta_j$ ranges over the $G(k+1)$-st roots of unity. Since we have controlled the size of polynomials of degree at most $G(k+1)$ at points which are $O(G(k+1)^{-1})$-spaced, Bernstein's inequality implies 
$$\max_{G(k) \leq s \leq G(k+1)} \|\Phi_s^*\|_\infty \leq C.$$
Therefore
$$\sup_n \|\Phi_n^*\|_\infty \leq C.$$
\end{proof}

Controlling the uniform norm of $\left|\sum_{j=1}^\infty \alpha_je^{i\gamma_j(\theta)}\right|$ is almost exactly the same problem encountered by Salem and Zygmund in \cite{salem:spt54}, but with a particular form of dependence between terms. The next Lemma dispenses with this dependence.

Define
$$A_k(\omega, \theta) \coloneqq \sum_{j=G(k)}^{G(k+1)} \alpha_je^{i\gamma_j(\theta)}$$
and
$$B_k(\omega) \coloneqq \sum_{j=G(k)}^{G(k+1)} \alpha_j.$$

\begin{lemma}\label{no_more_dependence}
For $p \in \mathbb{Z}_+$ and fixed $\theta \in [0, 2\pi)$,
we have
$$\mathbb{E}\left[|A_k(\omega, \theta)|^{p}\right] \leq \mathbb{E} \left[ |B_k(\omega)|^{p}\right].$$
\end{lemma}
\begin{proof}
One may show this by expanding both quantities and using the conditional independence. But in fact this Lemma is immediate, since the laws of $A_k$ and $B_k$ are identical under the assumption that the $\alpha_j$ are symmetrically distributed. 
\end{proof}

{\bf Remark.} Lemmas \ref{reduction_to_sum} and \ref{no_more_dependence} indicate that we can use the approach of Salem and Zygmund with very few modifications to prove Theorem \ref{polynomial-SZ}.

We haven't yet utilized the martingale nature of the sum
$$\sum_{m=l}^k \alpha_me^{i\gamma_m(\theta_j)}$$
but it does play an important role, in extending the estimates we obtain from particular lattice points $G(k)$ to the entire sequence of positive integers. The tool we use for this purpose is the Doob martingale inequality. Recall that a submartingale is defined similarly to a martingale; a sequence of integrable random variables $\{X_n\}$ is a discrete submartingale with respect to a filtration $\left\{\mathcal{G}_k\right\}$ on a probability space if $\{X_n\}$ is adapted to the filtration and satisfies the inequality
$$\mathbb{E}\left[ X_n | \mathcal{G}_k \right] \geq X_{n-1}.$$

\begin{lemma}[Doob]\label{doob}
	Let $\{X_t\}$ be a nonnegative submartingale in discrete time. Then for any $C>0$,
	$$\mathbb{P} \left[ \max_{0 \leq t  \leq T} X_t \geq C \right] \leq \frac{\mathbb{E}[X_T]}{C}.$$
\end{lemma}
A proof of this statement can be found in \cite{durrett:pte10}, where it is stated as Theorem 5.4.2.

\section{Proof of Theorem \ref{polynomial-SZ}}

\begin{proof}[Theorem \ref{polynomial-SZ}]
Recall we have set $G(j) = 2^{2^{j}}$.

By Lemma \ref{reduction_to_sum}, it suffices to show, for some summable sequence $\{\lambda_k\}$ and $\{\theta_j\}_{j=1}^{G(k+1)}$ the $G(k+1)$-st roots of unity, the sequence of events
$$E_k = \bigcup_{j=1}^{G(k+1)} \left\{ \max_{G(k) \leq s \leq G(k+1)} \left|\sum_{m=G(k)}^s \alpha_m e^{i \gamma_m(\theta_j)}\right| \geq \lambda_k\right\}$$
occurs only finitely often with probability 1.

First we reduce to a one-point estimate. Let
$$P_k \coloneqq \mathbb{P}\left[ \bigcup_{j=1}^{G(k+1)}\left\{ \max_{G(k) \leq s \leq G(k+1)} \left|\sum_{m=G(k)}^s \alpha_m e^{i \gamma_m(\theta_j)}\right| \geq \lambda_k\right\}\right].$$ 
Clearly
$$P_k \leq \sum_{j=1}^{G(k+1)} \mathbb{P}\left\{\max_{G(k) \leq s \leq G(k+1)} \left|\sum_{m=G(k)}^s \alpha_m e^{i \gamma_m(\theta_j)}\right| \geq \lambda_k\right\}.$$
Letting $\theta_j$ be a point at which the probability is maximized, we have  
$$P_k \leq G(k+1) \mathbb{P} \left\{\max_{G(k) \leq s \leq G(k+1)} \left|\sum_{m=G(k)}^s \alpha_m e^{i \gamma_m(\theta_j)}\right| \geq \lambda_k\right\}.$$

We now apply the Doob martingale inequality, Lemma \ref{doob}. By the conditional Jensen inequality, $$M_{G(k)\to s}\coloneqq \exp\left(t\left|\sum_{m=G(k)}^s \alpha_m e^{i \gamma_m(\theta_j)}\right|\right)$$ is a submartingale for any $t>0$. So by Lemma \ref{doob},
$$\mathbb{P} \left\{\max_{G(k) \leq s \leq G(k+1)} \left|\sum_{m=G(k)}^s \alpha_m e^{i \gamma_m(\theta_j)}\right| \geq \lambda_k\right\} = \mathbb{P}\left\{ \max_{G(k) \leq s \leq G(k+1)} M_{G(k) \to s} \geq e^{t\lambda_k}\right\} \leq \frac{\mathbb{E}\left[ e^{tA_k(\omega, \theta_j)}\right]}{e^{t\lambda_k}}$$ for any $t>0$, where as before
$$A_k(\omega, \theta) = \sum_{j=G(k)}^{G(k+1)} \alpha_je^{i\gamma_j(\theta)}.$$

So we wish to control
$$S_k \coloneqq \mathbb{E}\left[e^{tA_k}\right].$$
By Lemma \ref{no_more_dependence}, all moments of the random variable $A_k(\omega, \theta)$ for fixed $\theta$ are controlled by those of the random variable
$$B_k = \sum_{j=G(k)}^{G(k+1)} \alpha_j.$$
$B_k$ is sub-Gaussian with variance proxy at most $R_{G(k)} = \sum_{j \geq G(k)} |\alpha_j|^2$, as the sum of independent sub-Gaussian random variables. Therefore $tA_k$ is sub-Gaussian with variance proxy at most $t^2R_{G(k)}$. So
$$S_k \leq \exp\left(\frac{t^2R_{G(k)}}{2}\right)$$
and
$$P_k \leq G(k+1) \exp\left(\frac{t^2R_{G(k)}}{2}-t\lambda_k\right).$$
Minimizing the exponent in $t$, we find $t = \frac{\lambda_k}{R_{G(k)}}$ and we have
$$P_k \leq G(k+1) \exp \left(- \frac{\lambda_k^2}{2R_{G(k)}}\right).$$

Now we choose $\lambda_k$. We wish to simultaneously have
$$\sum_k P_k < \infty \quad \quad \quad \sum_k \lambda_k < \infty$$
We take
$$\lambda_k = 3\sqrt{\log(G(k+1))R_{G(k)}}.$$
Then
$$P_k \leq G(k+1)\exp \left(- \frac{\lambda_k^2}{2R_{G(k)}}\right) \leq \exp\left(\log G(k+1)-\frac{3}{2}\log(G(k+1)) \right) =\exp\left(-2^{k}\log2\right)$$
so that 
$$\sum_k P_k <\infty.$$
To estimate $\sum_k \lambda_k$ we use our assumption on $R_n$. By Cauchy condensation, we have the implications
$$\sum_n \frac{\sqrt{R_n}}{n\sqrt{\log n}} < \infty \implies \sum_j \frac{\sqrt{R_{2^j}}}{j} <\infty \implies \sum_k \sqrt{2^jR_{2^{2^j}}}<\infty$$
which exactly says that 
$$\sum_k \lambda_k < \infty$$
by our choice of $G(k)$.

So by the Borel-Cantelli Lemma, the events $$\bigcup_{j=1}^{G(k+1)} \left\{ \max_{G(k) \leq s \leq G(k+1)} \left|\sum_{m=G(k)}^s \alpha_m e^{i \gamma_m(\theta_j)}\right| \geq \lambda_k\right\}$$ occur infinitely often with probability 0, and by Lemma \ref{reduction_to_sum} we have Theorem \ref{polynomial-SZ}.

\end{proof}

\section{Proof of Theorem \ref{sharpness}}

We begin with an important Lemma. 

\begin{lemma}
	\label{blowup}
	Let $\{\alpha_j\}\in \mathbb{D}^\infty$ and $\{\gamma_j(\theta)\}$ the associated Pr\"{u}fer phases. Assume that these satisfy the following conditions:
	\begin{enumerate}
	\item $\left|\gamma_j(\theta)-\gamma_j(0)-(j+1)\theta\right|\leq \frac{\pi}{12}$
	\item $\sum_n |\alpha_n| = \infty$
	\item For some $T \geq 12$, $\alpha_j=0$ unless $j=T^k$ for $k \geq 0$.
	\end{enumerate}
	Then
	$$\lim_{j \to \infty} \sup_{\theta \in [0,2\pi)}  \left| \sum_{n=0}^j \alpha_n e^{i\gamma_n(\theta)}\right| = \infty.$$
\end{lemma}
\begin{proof}
	Let 
	$$\Lambda_j=\left\{\theta \in [0, 2\pi): |\alpha_{T^j}| \leq 2\Re\left(\alpha_{T^j}e^{i\gamma_{T^j}(\theta)}\right)\right\}.$$
	Notice it suffices to show $\bigcap_{j \geq 0} \Lambda_j \neq \emptyset$. We will in fact show by induction
	\begin{equation}\label{ind_hyp} \bigcap_{j=0}^N \Lambda_j \text{ contains an interval of measure } \frac{2\pi+\frac{\pi}{6}}{T^{N+1}+1}.
	\end{equation}
	For $N=0$, condition (1) will yield \eqref{ind_hyp} but in fact it is not necessary. Since $\Phi_0=\Phi_0^*=1$, $\gamma_0(\theta) = \theta$ and $\Lambda_0$ contains an interval of measure $\frac{\pi}{3}\geq \frac{2\pi+\frac{\pi}{6}}{T+1}$.
	
	Assume \eqref{ind_hyp} holds for $N-1$. Denote by $\theta_1 \in [0, 2\pi)$ the argument of $\alpha_{T^N}$, that is,
	$$\alpha_{T^N} = |\alpha_{T^N}|e^{i\theta_1}.$$
	By assumption (1) along with the inductive hypothesis and continuity of the Pr\"{u}fer phase, $e^{i\gamma_{T^N}(\theta)+\theta_1}$ must run through $\mathbb{T}$ as $\theta$ runs through some interval contained in $\cap_{j=0}^{N-1} \Lambda_j$. Therefore there is an interval $I_N \subseteq \bigcap_{j=0}^{N-1} \Lambda_j$ so that 
	$$J_N = \{e^{i\gamma_{T^N}(\theta)+\theta_1}: \theta \in I_N\} \subseteq \left[e^{-i\pi/3},e^{i\pi/3}\right] \text{ and } |J_N| \geq \frac{\pi}{3}.$$
	
	Since $\gamma_{T^N}(\theta) = (T^N+1)\theta + C + f(\theta)$ for $|f(\theta)| \leq \frac{\pi}{12}$ by condition (1), 
	$$|I_N| \geq \frac{\pi}{6(T^N+1)}$$
	and by definition $$I_N \subseteq \bigcap_{j=0}^N \Lambda_j.$$
	So if $T$ satisfies \begin{equation}\label{T-condition}
	\frac{2\pi + \frac{\pi}{6}}{T^{N+1}+1} \leq \frac{\pi}{6(T^N+1)}
	\end{equation}
	we have the Lemma.
	
	For $T>1$, the function $\frac{T^{N+1}+1}{T^N+1}$ is minimized at $N=0$, so if $T \geq 12$ then $\frac{T^{N+1}+1}{T^N+1}\geq \frac{13}{2}$ and $T$ satisfies \eqref{T-condition}.

\end{proof}

\begin{proof}[Proof (Theorem \ref{sharpness}).]
Let $\epsilon>0$ be a control parameter. Let $T \in \mathbb{Z}, T \geq 12$, and $\psi$ be a nonnegative function on $\mathbb{N}$. We let, for $n \geq 2$,
$$a_{T^n} = \frac{\epsilon}{n\psi(n)}$$
and all other $a_j = 0$. We assume $\psi(n) \to \infty$ monotonically with $n$, but
$$\sum_n \frac{1}{n\psi(n)} = \infty.$$
Let $\omega_n=e^{iu_n}$ where $\{u_n\}$ is a sequence of independent random variables uniform in $[0, 2\pi)$, and $\alpha_n=a_n\omega_n$.
If we assume that $\{a_n\}$ satisfies the conditions of Theorem \ref{sharpness}, it is sufficient to show
$$\sup_N  \sup_{\theta \in [0, 2\pi)}\left|\sum_{n=0}^N \alpha_{T^n}e^{i\gamma_{T^n}(\theta)}\right| = \infty$$
since $\{\alpha_j\} \in \ell^2(\mathbb{Z}_+)$ and $\log \Phi_n^* = \sum_{j=0}^{n-1} \alpha_je^{i\gamma_j(\theta)} + O\left(\sum_{j=0}^{n-1} |\alpha_j|^2\right)$.

As established in \cite{killip:esc09} and stated in Section 2.3 of \cite{chhaibi:omc16}, the Pr\"{u}fer phases satisfy the following recursion:
\begin{equation}
\label{prufer-recursion}
\gamma_0(\theta)=\theta,$$
$$\gamma_{m+1}(\theta)-\gamma_{m+1}(0) $$
$$= \gamma_m(\theta)-\gamma_m(0) + \theta -2\left( \Im \left(1-\alpha_me^{i\gamma_m(\theta)}\right)-\Im \log\left(1-\alpha_me^{i\gamma_m(0)}\right)\right).
\end{equation}

By \eqref{prufer-recursion}, 
$$\left|\gamma_m(\theta)-\gamma_m(0)-(m+1)\theta\right| \leq 2\left| \sum_{n \leq m-1} \sum_{k=1}^{\infty} \frac{(\alpha_n)^k}{k}(e^{ik(\gamma_n(\theta))}-e^{ik\gamma_n(0)})\right| $$
$$\leq O(\epsilon) + 2\left| \sum_{n=1}^m \alpha_n(e^{i\gamma_n(\theta)}-e^{i\gamma_n(0)})\right|.$$

Suppose there is a random constant $C(\omega)$ such that with probability 1,  \begin{equation}
\label{assumption_towards_contradiction}
\sup_{j \geq 1} \sup_{\theta \in [0, 2\pi)} \left|\sum_{n=1}^j \frac{\omega_n}{n\psi(n)}e^{i\gamma_{T^n}(\theta)}\right|\leq C(\omega).
\end{equation}
Suppose further that $C(\omega)$ is the minimal such constant.

Recall that the moments of $\sum_{n=1}^j \frac{\omega_n}{n\psi(n)}e^{i\gamma_{T^n}(0)}$ are controlled by those of $\sum_{n=1}^j \frac{\omega_n}{n\psi(n)}$ by symmetry of $\omega_n$. Then noting that $\left\{\frac{1}{n\psi(n)}\right\} \in \ell^2$ shows $\left| \sum_{n=1}^j \frac{\omega_n}{n\psi(n)}e^{i\gamma_{T^n}(0)}  \right| \leq C(\omega)$ with probability 1, and so our assumption implies $\left| \sum_{n=1}^j \alpha_{T^n}\left(e^{i\gamma_{T^n}(\theta)}-e^{i\gamma_{T^n}(0)}\right)\right| \leq \epsilon C(\omega)$ uniformly in $\theta$ and $j$ for $C(\omega) \in L^\infty(\Omega)$. 

Let $\Omega_0 = \cup_n \left\{ C(\omega) \leq n\right\}$, so that $\mathbb{P}\left(\Omega_0\right)=1$. Here we use our control over $\epsilon$. Since all estimates scale with $\epsilon$ we may assume that for any fixed $\epsilon_1>0$,
$$\Omega_1 = \left\{\omega \in \Omega: \sup_{j \geq 1} \sup_{\theta \in [0, 2\pi)} \left| \sum_{n=1}^j \alpha_{T^n}\left(e^{i\gamma_{T^n}(\theta)}-e^{i\gamma_{T^n}(0)}\right)\right| \leq \epsilon_1 \right\} \text{ satisfies } \mathbb{P}\left(\Omega_1\right) \geq \frac{1}{2}.$$
So we may take $\epsilon_0>0$ so that for $a_{T^n} = \frac{\epsilon_0}{n\psi(n)}$ and $\omega \ in \Omega_1$, $$\gamma_j(\theta) = \gamma_j(0) + (j+1)\theta + f(\theta), \quad |f(\theta)| \leq \frac{\pi}{12}.$$

But by Lemma \ref{blowup}, under this choice of $\{a_n\}$ and for $\omega \in \Omega_1$,
$$\sup_{j \geq 1} \sup_{\theta \in [0, 2\pi)} \left|\sum_{n=1}^j \alpha_ne^{i\gamma_n(\theta)}\right|=\infty.$$

This is a contradiction, and so with positive probability \begin{equation}
\label{conclusion}
\sup_{\theta \in [0, 2\pi)} \sup_{j \geq 0} \left|\sum_{n=1}^j \alpha_ne^{i\gamma_n(\theta)}\right|=\infty.
\end{equation}

\eqref{conclusion} is a tail event, so by Kolmogorov's 0-1 law we must in fact have \eqref{conclusion} with probability 1. Therefore almost surely
$$\sup_n \left\|\Phi_n^*\right\|_\infty = \infty.$$
Now we must ensure that the conditions of Theorem \ref{sharpness} are satisfied. The argument to follow can essentially be found in \cite{salem:spt54}, Remarks on Theorem (5.1.5).

$$R_{T^n} = \sum_{k=n}^{\infty} \frac{1}{k^2\psi(k)^2} \leq \sum_{k=n}^{\infty} \frac{1}{k^2\psi(n)^2} \lesssim \frac{1}{n\psi^2(n)}$$
as $\psi$ is monotonic, and so for any $\beta(n) \to \infty$ monotonically in $n$,
$$\sqrt{\frac{R_{T^k}}{k}}\frac{1}{\beta(k)} \lesssim \frac{1}{k\psi(k)\beta(k)}.$$
No matter how slowly $\beta$ increases, there is $\psi(k)$ so that
$$\sum_k \frac{1}{k\psi(k)} = \infty, \quad \sum_k \frac{1}{k\psi(k)\beta(k)} < \infty.$$
Letting $\nu(T^k) = \beta(k)$, generalized Cauchy condensation implies
$$\sum_n \frac{\sqrt{R_n}}{\nu(n)n\sqrt{\log n}} < \infty \iff \sum_k \sqrt{\frac{R_{T^k}}{k}} \frac{1}{\beta(k)}<\infty$$
so the condition in Theorem \ref{polynomial-SZ} is the best possible of its kind.
\end{proof}

\textbf{Acknowledgement.} 

The research of the author was supported by the RTG grant NSF-DMS-1147523

\bibliographystyle{amsplain}

\end{document}